\documentclass[11pt]{article}
\usepackage[all]{xy}
\usepackage{amsfonts}
\usepackage{amsthm}
\usepackage{enumerate}
\usepackage{graphicx}
\usepackage{mathrsfs}
\usepackage{bm}
\usepackage{authblk}
\usepackage{cite}
\usepackage[colorlinks,linkcolor=red,citecolor=red]{hyperref}
\usepackage{amssymb,amsmath} 
\usepackage{makecell}
\usepackage{tikz}
\usepackage{pgf}
\usepackage[T1]{fontenc}
\usepackage[utf8]{inputenc}
\usepackage{tikz}
\usetikzlibrary{patterns}
\usepackage{pgffor}
\usepackage{pgfcalendar}
\usepackage{pgfpages}
\usepackage{shuffle,yfonts}
\usepackage{mathtools}
\DeclareFontFamily{U}{shuffle}{}
\DeclareFontShape{U}{shuffle}{m}{n}{ <-8>shuffle7 <8->shuffle10}{}

\newcommand{\bfs}{{\boldsymbol{\sl{s}}}}

 \allowdisplaybreaks
\usetikzlibrary{arrows,shapes,chains}
 \allowdisplaybreaks

\catcode`!=11
\let\!int\int \def\int{\displaystyle\!int}
\let\!lim\lim \def\lim{\displaystyle\!lim}
\let\!sum\sum \def\sum{\displaystyle\!sum}
\let\!sup\sup \def\sup{\displaystyle\!sup}
\let\!inf\inf \def\inf{\displaystyle\!inf}
\let\!cap\cap \def\cap{\displaystyle\!cap}
\let\!max\max \def\max{\displaystyle\!max}
\let\!min\min \def\min{\displaystyle\!min}
\let\!frac\frac \def\frac{\displaystyle\!frac}
\catcode`!=12

\let\oldsection\section
\renewcommand\section{\setcounter{equation}{0}\oldsection}

\allowdisplaybreaks

\def\Z{\mathbb{Z}}

\def\ze{\zeta}

\theoremstyle{plain}
\newtheorem{thm}{Theorem}[section]
\newtheorem{lem}[thm]{Lemma}
\newtheorem{cor}[thm]{Corollary}

\newtheorem{pro}[thm]{Proposition}
\theoremstyle{definition}
\newtheorem{defn}{Definition}[section]

\newtheorem{exa}[thm]{Example}

\begin{document}
\title{\bf On a Class of Berndt-type Integrals and Related Barnes Multiple Zeta Functions}
\author{
		{ Xiang Chen$^{a,}$\thanks{Email: 3011270372@ahnu.edu.cn},\quad Ce Xu${}^{a,}$\thanks{Email: cexu2020@ahnu.edu.cn}\quad and\quad Jianing Zhou$^{b}$\thanks{Email: 202421511272@smail.xtu.edu.cn}\\
		\small $^{a}$ School of Mathematics and Statistics, Anhui Normal University, \\ Wuhu 241002, P.R.C.  \\
		\small $^{b}$ School of Mathematics and Computational Science, Xiangtan University, \\ Xiangtan 411105, P.R.C. }}
		\date{}
\maketitle
\begin{abstract}
This paper investigates a class of special Berndt-type integral calculations where the integrand contains only hyperbolic cosine functions. The research approach proceeds as follows:
Firstly, through contour integration methods, we transform the integral into a Ramanujan-type hyperbolic infinite series. Subsequently, we introduce a $\theta$-parameterized auxiliary function and apply the residue theorem from complex analysis to successfully simplify mixed-type denominators combining hyperbolic cosine and sine terms into a normalized Ramanujan-type hyperbolic infinite series with denominators containing only single hyperbolic function terms.
For these simplified hyperbolic infinite series, we combine properties of Jacobi elliptic functions with composite analytical techniques involving Fourier series expansion and Maclaurin series expansion. This ultimately yields an explicit expression as a rational polynomial combination of $\Gamma(1/4)$ and $\pi^{-1/2}$.
Notably, this work establishes a connection between the integral and Barnes multiple zeta functions, providing a novel research pathway for solving related problems.
\end{abstract}



\section{Introduction}

Recently, Xu and Zhao \cite{XZ2024} defined anomalous integrals of the form:
\begin{align}\label{defn-Berndttypeintegrals}
{\rm BI}^{\pm}(s,m):=\int_{0}^{\infty} \frac{x^{s-1} \mathrm{d}x}{(\cos x \pm \cosh x)^{m}}
\end{align}
where $s\geq 1$ and $m\geq 1$ if the denominator has `$+$' sign, and $s\geq 2m+1$ otherwise. They termed these as \emph{$m$-th order Berndt-type integrals}, since Berndt \cite{Berndt2016} was the first to systematically study such integrals for the case $m=1$. But in fact, the study of such integrals can be traced back to over a century ago, when the renowned Indian mathematician Ramanujan posed the following integral problem in the \emph{Indian Journal of Pure and Applied Mathematics} \cite[pp. 325-326]{Rama1916}:
\begin{align}\label{defn-Ramanujantypeintegrals}
\int_{0}^{\infty} \frac{\sin (n x)}{x(\cos x+\cosh x)} \mathrm{d}x=\frac{\pi}{4} \quad(n \in \mathbb{N}).
\end{align}
Wilkinson \cite{W1916} provided the first rigorous proof in the fourth year after the problem was posed. Subsequently, other proofs were also discovered. For instance, Berndt \cite[Thm. 4.1]{Berndt2016} utilized the method of contour integration to prove this integral equality proposed by Ramanujan. Meanwhile, Berndt also employed methods of contour integration and  Jacobi elliptic functions to study the case of $m=1$ in the integral \eqref{defn-Berndttypeintegrals}. He demonstrated that certain types of integrals under this condition could be expressed as a combination of rational polynomials in $\Gamma(1/4)$ and $\pi^{-1/2}$, and provided specific examples. However, no general structural formulation was given. In a recent series of papers \cite{RXYZ2024,RXZ2023,XZ2023,XZ2024,ZhangRui2024}, Rui, Xu, Yang, Zhao, and Zhang extended Berndt's methods to establish structural theorems for arbitrary-order Berndt integrals: for all integers $m\geq 1$ and $p\geq [m/2]$,
\begin{align*}
{\rm BI}^+(4p+2,m)\in\mathbb{Q}[X,Y]\quad\text{and}\quad {\rm BI}^-(a,m)\in\mathbb{Q}[X,Y] \quad (0<a-2m\equiv 1 \pmod{4}),
\end{align*}
where $X:=\Gamma^4(1/4)$ and $Y:=\pi^{-1}$. The optimal order range for the variables $X$ and $Y$ in the polynomials has been explicitly determined (see \cite[Thms 1.1 and 1.2]{XZ2024}). Recently, Pan and Wang \cite{PanWang2025} utilized the method of contour integration to study integrals analogous to \eqref{defn-Ramanujantypeintegrals}, and obtained the following two results:
\begin{align*}
&\int_0^\infty\frac{x^b\sin(nx)}{\cos x-\cosh x}\mathrm{d}x=0\quad (b\equiv -1\ ({\rm mod}\ 4)\quad \text{and}\quad n\in \Z),\\
&\int_0^\infty \frac{x^{4p+1}\sin x}{\cos x-\cosh x}\mathrm{d}x=(-1)^{p+1}2^{2p}\pi^{4p+2}\frac{B_{4p+2}}{2p+1}\quad(0\leq p\in \Z),
\end{align*}
where $B_{m}$ is \emph{Bernoulli number} which is defined by the generating function
\begin{equation*}
 \frac{xe^x }{e^x-1}=\sum_{n=0}^\infty \frac{B_n}{n!}x^n.
\end{equation*}
Surprisingly, when the complete elliptic integrals of the first kind $K(x)$ and its complement $K:=K(1-x)$ appear in the denominator of the integrand in \eqref{defn-Berndttypeintegrals}, the integral can be expressed as higher-order derivatives of powers of certain Jacobi elliptic functions after division. In this area of research, the pioneering work was carried out by Ismail and Valent \cite{Ismail1998}, who discovered this landmark result:
\begin{align}\label{A.KuznetsovInt}
\int_{-\infty}^{\infty}{\frac{\mathrm{d}t}{\cos \left( K\sqrt{t} \right) +\cosh \left( K^{\prime}\sqrt{t} \right)}=2},
\end{align}
where the \emph{complete elliptic integrals of the first kind} $K$ is defined as follows (see Whittaker and Watson \cite{W1916}):
\begin{equation*}
    K:=K\left(k^2\right)=\int_{0}^{\pi/2} \frac{\mathrm{d}\varphi}{\sqrt{1 - k^2 \sin^2 \varphi}}=\frac{\pi}{2}\,{}_2F_1\left(\frac{1}{2},\frac{1}{2};1;k^2\right).
\end{equation*}
The ${}_2F_1(a,b;c;x)$ denotes the \emph{Gaussian hypergeometric function} defined by
\begin{equation*}
    {}_2F_1(a,b;c;x):=\sum_{n = 0}^{\infty} \frac{(a)_n(b)_n}{(c)_n}\frac{x^n}{n!}\quad (a,b,c\in\mathbb{C}),
\end{equation*}
where $\left( a \right) _0:=1$ and $\left( a \right)_n:=\frac{\Gamma \left( a+n \right)}{\Gamma \left( a \right)}=\prod_{j=0}^{n-1}{\left( a+j \right)}\quad (n\in \mathbb{N}),
$ and $ \Gamma(x)$ is the \emph{Gamma function} defined by
\begin{equation*}
    \Gamma(z):=\int_{0}^{\infty}e^{-t}t^{z - 1}\mathrm{d}t\quad (\Re(z)>0).
\end{equation*}
Similarly, the \emph{complete elliptic integral of the second kind} is denoted by (see Whittaker and Watson \cite{W1916})
\begin{equation*}
E:=E(x):=E(k^2):=\int\limits_{0}^{\pi/2} \sqrt{1-k^2\sin^2\varphi}\mathrm{d}\varphi=\frac {\pi}{2} {_2}F_{1}\left(-\frac {1}{2},\frac {1}{2};1;k^2\right).
\end{equation*}
Meanwhile, Kuznetsov \cite{K2017} further applied methods of contour integration and theta functions to extend the above integral \eqref{A.KuznetsovInt} to a more general form, obtaining the following remarkable result:
\begin{align}\label{Ims-Kuz+}
		\frac1{2}\int_{-\infty}^\infty \frac{t^n \mathrm{d}t}{\cos(K\sqrt{t})+\cosh(K^{\prime}\sqrt{t})}=(-1)^n \frac{d^{2n+1}}{du^{2n+1}} \frac{{\rm sn} (u,k)}{{\rm cd}(u,k)}|_{u=0},
	\end{align}
where $x=k^2\ (0<k<1)$ and ${\rm sn} (u,k)$ is the \emph{Jacobi elliptic function} defined through the inversion of the elliptic integral
	\begin{align}
		u=\int_0^\varphi \frac{\mathrm{d}t}{\sqrt{1-k^2\sin^2 t}}\quad (0<k^2<1),
	\end{align}
	that is, ${\rm sn}(u):=\sin \varphi$. As before, $k$ is referred to as the elliptic modulus. We also write $\varphi={\rm am}(u,k)={\rm am}(u)$ and call it the \emph{Jacobi amplitude}.
	The Jacobi elliptic function ${\rm cd} (u,k)$ can be defined as follows
	\begin{align*}
		& {\rm cd}(u,k):=\frac{{\rm cn}(u,k)}{{\rm dn}(u,k)},
	\end{align*}
	where ${\rm cn}(u,k):=\sqrt{1-{\rm sn}^2 (u,k)}$ and ${\rm dn} (u,k):=\sqrt{1-k^2{\rm sn}^2(u,k)}$. Recently, Bradshaw and Vignat \cite{BV2024} extended Kuznetsov's approach, investigating and establishing explicit relationships between integrals with denominators involving the difference of cosine and hyperbolic cosine functions and Jacobi elliptic functions:
for \( n \in \mathbb{N} \cup \{0\} \),
\begin{equation}
    \int_{\mathbb{R}} \frac{x^{n + 1}}{\cos(K\sqrt{x}) - \cosh(K^{\prime}\sqrt{x})}\mathrm{d}x = \left. (-1)^{n + 1} 8 \frac{d^{2n + 1}}{du^{2n + 1}} \frac{\text{sn}^2(u, k)}{\text{cd}^2(u, k) \text{sd}(2u, k)} \right|_{u = 0},
\end{equation}
where the Jacobi elliptic function \( \mathrm{sd}(u) \equiv \mathrm{sd}(u, k) \) is defined by
\begin{equation}
    \mathrm{sd}(u, k) := \frac{\mathrm{sn}(u, k)}{\mathrm{dn}(u, k)}.
\end{equation}
Another surprising result is that Bradshaw and Vignat \cite{BV2024} discovered this type of Berndt integral is closely related to the Barnes multiple zeta function. In particular, they provided the following explicit relationships: (\cite[Prop. 2]{BV2024})
\begin{align*}
&\int_0^\infty \frac{x^{s}\mathrm{d}x}{(\cos x-\cosh x)^m}=2^m \Gamma(a+1)\ze_{2m}(s+1,m|(1+i,1-i)^m)\quad (s\geq 2m,m\geq 1),\\
&\int_0^\infty \frac{x^{s}\mathrm{d}x}{(\cos x+\cosh x)^m}=2^m \Gamma(a+1){\bar \ze}_{2m}(s+1,m|(1+i,1-i)^m)\quad (s\geq 0,m\geq 1),
\end{align*}
where $\bfs^n$ means the string $\bfs$ is repeated $n$ times. Here for positive real numbers $a_1,\ldots,a_N$, the \emph{Barnes multiple zeta function} and \emph{alternating Barnes multiple zeta function} are defined by (\cite{KMT2023,R2000})
\begin{align*}
&\ze_N(s,\omega|a_1,\ldots,a_N):=\sum_{n_1\geq0,\ldots,n_N\geq 0} \frac{1}{(\omega+n_1a_1+\cdots+n_Na_N)^s}\quad (\Re(\omega)>0,\Re(s)>N),\\
&{\bar \ze}_N(s,\omega|a_1,\ldots,a_N):=\sum_{n_1\geq0,\ldots,n_N\geq 0} \frac{(-1)^{n_1+\cdots+n_N}}{(\omega+n_1a_1+\cdots+n_Na_N)^s}\quad (\Re(\omega)>0,\Re(s)> N-1),
\end{align*}
respectively.

In this paper, we focus on Berndt-type integrals of the form:
\begin{equation}
    \int_{0}^{\infty} \frac{x^s \mathrm{d}x}{\left[ \cosh(2x) - \cos(2x) \right] \left[ \cosh x - \cos x \right]}\quad (\Re(s)\geq 4).
\end{equation}
Here, we focus on Berndt-type integrals with integrands consisting purely of (hyperbolic) cosine functions, where each integrand is expressed as a product of two specific terms: (i) the sum of a cosine and a hyperbolic cosine function, and (ii) their difference. Through the application of contour integration and series expansions of Jacobi elliptic functions, we establish structural theorems and obtain explicit evaluations for these Berndt-type integrals. Specifically, we shall prove the following results
(also see Theorem \ref{thmEA}): let \( \Gamma = \Gamma(1/4) \), for \( m \in \mathbb{N}\setminus \{1\} \),
\begin{multline}\label{mainformula}
    \int_{0}^{\infty} \frac{x^{4m - 3} \mathrm{d}x}{\left[ \cosh(2x) - \cos(2x) \right] \left[ \cosh x - \cos x \right]} \in \mathbb{Q} \frac{\Gamma^{8m - 8}}{\pi^{2m - 2}} + \frac{\mathbb{Q}}{\sqrt{2}} \frac{\Gamma^{8m - 6}}{\pi^{(4m - 3)/2}} \\
    + \mathbb{Q} \frac{\Gamma^{8m-4}}{\pi^{2m-1}} + \frac{\mathbb{Q}}{\sqrt{2}} \frac{\Gamma^{8m - 2}}{\pi^{(4m + 1)/2}} + \mathbb{Q} \frac{\Gamma^{8m }}{\pi^{2m + 2}}.
\end{multline}
Additionally, we present the specific coefficients preceding each term and establish explicit relationships between the Berndt-type integrals and Barnes multiple zeta functions. In addition, we further clarify their connection with Barnes multiple zeta functions, thereby providing new closed-form evaluations and insightful perspectives (see Theorem \ref{thmFB}).

\section{Some Definitions and Lemmas}

This section lays the foundational mathematical framework for our subsequent analysis, employing rigorously constructed definitions and technical lemmas.
\begin{defn}
	Let \( s \in \mathbb{C} \) and \( a, b, \theta \in \mathbb{R} \) with \( |\theta| < 3b\pi \) and \( a, b \neq 0 \). Define
\[
\begin{aligned}
S_1(s, \theta; a, b) & := \frac{\pi^4 \sinh(\theta s)}{\sin(a\pi s) \sinh(b\pi s) \cosh^2(b\pi s)},\quad S_2(s, \theta; a, b)  := \frac{\pi^4 \cosh(\theta s)}{\cos(a\pi s) \sinh(b\pi s) \cosh^2(b\pi s)}, \\
S_3(s, \theta; a, b) & := \frac{\pi^4 \sinh(\theta s)}{\sin(a\pi s) \sinh^2(b\pi s) \cosh(b\pi s)}, \quad S_4(s, \theta; a, b) := \frac{\pi^4 \sinh(\theta s)}{\sin(a\pi s) \sinh^3(b\pi s) }.
\end{aligned}
\]
\end{defn}

\begin{defn}
	Let \(m \in \mathbb{N}\) and \(p \in \mathbb{Z}\). Define
	\begin{align*}
		C_{p,m}(y) &:=\sum_{n=1}^{\infty} \frac{(-1)^{n} n^{p}}{\sinh (n y) \cosh ^{m}(n y)}, \quad X_{p,m}(y) :=\sum_{n=1}^{\infty} \frac{(-1)^{n} n^{p}}{\sinh ^{m}(n y)}, \\
		DX_{p,m}(y) &:=\sum_{n=1}^{\infty} \frac{(-1)^{n} n^{p} \cosh(ny)}{\sinh ^{m}(n y)}, \quad
		C'_{p,m}(y) :=\sum_{n=1}^{\infty} \frac{(-1)^{n}(2 n-1)^{p-1}}{\sinh ((2 n-1) y / 2) \cosh ^{m}((2 n-1) y / 2)}, \\
		\overline{C}_{p,m}(y) &:=\sum_{n=1}^{\infty} \frac{(-1)^{n}(2 n-1)^{p}}{\sinh ^{m}((2 n-1) y / 2) \cosh ((2 n-1) y / 2)}, \quad
		T_{p,m}(y) :=\sum_{n=1}^{\infty} \frac{(-1)^{n}(2 n-1)^{p-1}}{\sinh ^{m}((2 n-1) y / 2)}, \\
		DT_{p,m}(y) &:=\sum_{n=1}^{\infty} \frac{(-1)^{n}(2 n-1)^{p} \cosh ((2 n-1) y / 2)}{\sinh ^{m}((2 n-1) y / 2)}, \quad
		X'_{p,m}(y) :=\sum_{n=1}^{\infty} \frac{(-1)^{n}(2 n-1)^{p}}{\cosh ^{m}((2 n-1) y / 2)}, \\
		DX'_{p,m}(y) &:=\sum_{n=1}^{\infty} \frac{(-1)^{n}(2 n-1)^{p+1} \sinh ((2 n-1) y / 2)}{\cosh ^{m}((2 n-1) y / 2)}, \quad
		B_{p,m}(y) :=\sum_{n=1}^{\infty} \frac{(-1)^{n} n^{p-1}}{\cosh ^{m}(n y)}, \\
		DB_{p,m}(y) &:=\sum_{n=1}^{\infty} \frac{(-1)^{n} n^{p} \sinh (n y)}{\cosh ^{m} n y}.
	\end{align*}
\end{defn}

\begin{lem}\emph{(\cite{PB1998})}\label{lemAB}
  Let \(\xi(s)\) be a kernel function and let \(r(s)\) be a function which is \(O(s^{-2})\) at infinity. Then
\begin{equation}
\sum_{\alpha \in O} \operatorname{Res}(r(s) \xi(s), s = \alpha) + \sum_{\beta \in S} \operatorname{Res}(r(s) \xi(s), s = \beta) = 0,
\end{equation}
where \(S\) denotes the set of poles of \(r(s)\), and \(O\) denotes the set of poles of \(\xi(s)\) that are not poles of \(r(s)\). Here, \(\operatorname{Res}(r(s), s = \alpha)\) stands for the residue of \(r(s)\) at \(s = \alpha\). The kernel function \(\xi(s)\) is meromorphic throughout the entire complex plane and satisfies \(\xi(s) = o(s)\) over an infinite family of circles \(|s| = \rho_{k}\) with \(\rho_{k} \to \infty\).
\end{lem}

In fact, the Lemma \ref{lemBA} above can be entirely rewritten in the following unified form:
\begin{lem}\label{lemBA} Let \(f(s)\) be a meromorphic function on the complex plane such that \(f(s)=o(s^{-1})\) as $s\rightarrow \infty$. Then
\begin{equation}
\sum_{\alpha \in E} \operatorname{Res}(f(s), s = \alpha) = 0,
\end{equation}
where \(E\) denotes the set of poles of \(f(s)\).
\end{lem}
\begin{proof}
The proof of this lemma is completely analogous to that of Lemma \ref{lemAB} in Reference \cite{PB1998}, and therefore we omit it here.
\end{proof}

To better describe the main results of our article, we adopt the notation used by Ramanujan, defining $x$, $y$, and $z$ by the following relations:
\begin{equation} \label{eveBB}
    x:=k^2,\ y(x):=\pi K^{\prime}/K,\ q\equiv q(x):=e^{-y},\ z:=z(x)=2K/\pi,\ z'=dz/dx.
\end{equation}
Clearly, taking the $n$-th derivative with respect to $z$ will yield:
\begin{equation*}
    \frac{\mathrm{d}^n z}{\mathrm{d}x^n}=\frac{(1/2)_n^2}{n!}\,{}_2F_1\left(\frac{1}{2}+n,\frac{1}{2}+n,1 + n;x\right).
\end{equation*}
And using the known result for the hypergeometric series (\cite{A2000})
\begin{equation*}
{}_2F_1\left(a,b;\frac{a + b + 1}{2};\frac{1}{2}\right)=\frac{\Gamma\left(\frac{1}{2}\right)\Gamma\left(\frac{a + b + 1}{2}\right)}{\Gamma\left(\frac{a + 1}{2}\right)\Gamma\left(\frac{b +1}{2}\right)},
\end{equation*}
it follows that
\begin{equation}
    \left. \frac{d^n z}{dx^n} \right|_{x=1/2}=\frac{(1/2)_n^2 \sqrt{\pi}}{\Gamma\left(\frac{n}{2}+\frac{3}{4}\right)}.
\end{equation}
In particular, taking \( x = 1/2\) and \( n = 0,1 \) in \eqref{eveBB}, we obtain that
\begin{equation*}
y = \pi, \quad z\left( \frac{1}{2} \right) = \frac{ \Gamma^2(1/4)}{2\pi^{3/2}}, \quad z'\left( \frac{1}{2} \right) = \frac{4\sqrt{\pi}}{ \Gamma^2(1/4)}.
\end{equation*}
Furthermore, there exists the following derivative relationship among $x$, $y$ and $z$ (see \cite[pp. 120, Entry 9(i)]{B1991}):
\begin{equation} \label{eveBD}
\frac{dx}{dy} = -x(1 - x)z^2.
\end{equation}

\begin{lem}\emph{(\cite[Thm. 2.4.]{XuZhao-2024})}\label{lemBB}
 Let \(x, y, z\) and \(z'\) satisfy \eqref{eveBB}. Given the formula \(\Omega(x, e^{-y}, z, z') = 0\), we have the transformation formula
\begin{equation}\label{eveBE}
\Omega\left(1 - x, e^{-\pi^{2} / y}, y z / \pi, \frac{1}{\pi}\left(\frac{1}{x(1 - x) z} - y z'\right)\right) = 0.
\end{equation}
\end{lem}

\begin{lem}\emph{(\cite[Lem. 1.2.]{X2018})}\label{lemBC}
 Let \(n\) be an integer, then the following formulas hold:
\begin{align}
\frac{\pi}{\sin(\pi s)} &\stackrel{s \to n}{=} (-1)^n \left( \frac{1}{s - n} + 2\sum_{k=1}^{\infty} \overline{\zeta}(2k)(s - n)^{2k - 1} \right), \label{BF}\\
\frac{\pi}{\sinh(\pi s)} &\stackrel{s \to ni}{=} (-1)^n \left( \frac{1}{s - ni} + 2\sum_{k=1}^{\infty} (-1)^k \overline{\zeta}(2k)(s - ni)^{2k - 1} \right),\label{BG}\\
\frac{\pi}{\cos(\pi s)} &\stackrel{s \to n - 1/2}{=} (-1)^n \left\{ \frac{1}{s - \frac{2n - 1}{2}} + 2\sum_{k=1}^{\infty} \overline{\zeta}(2k) \left( s - \frac{2n - 1}{2} \right)^{2k - 1} \right\}, \\
\frac{\pi}{\cosh(\pi s)} &\stackrel{s \to (n - 1/2)i}{=} (-1)^n i \left\{ \frac{1}{s - \frac{2n - 1}{2}i} + 2\sum_{k=1}^{\infty} (-1)^k \overline{\zeta}(2k) \left( s - \frac{2n - 1}{2}i \right)^{2k - 1} \right\},
\end{align}
where $\zeta(s)$ and \(\overline{\zeta}(s)\) denote the \emph{Riemann zeta function} and \emph{alternating Riemann zeta function}, respectively, which are defined by
\[
\zeta(s):=\sum_{n=1}^\infty \frac1{n^s}\quad (\mathfrak{R}(s) > 1)\quad\text{and}\quad\overline{\zeta}(s) := \sum_{n=1}^{\infty} \frac{(-1)^{n - 1}}{n^s} \quad (\mathfrak{R}(s) > 0).
\]
\end{lem}

\begin{lem}\emph{(\cite[Eqs. (4.4)-(4.7)]{XZ2023})}\label{lemBD}
 Let \(n\) be an integer. Then we have
\begin{align}
\frac{(-1)^{n}}{\cosh \left(\frac{1+i}{2} z\right)} &\stackrel{z \to \tilde{n} \pi(1+i)}{=} 2i\left\{ \begin{array}{c}
\frac{1}{1+i} \cdot \frac{1}{z-\tilde{n} \pi(1+i)} \\
+\sum_{k=1}^{\infty} (-1)^{k} \frac{\overline{\zeta}(2k)}{\pi^{2k}} \left( \frac{1+i}{2} \right)^{2k-1} (z-\tilde{n} \pi(1+i))^{2k-1}
\end{array} \right\},\\
\frac{(-1)^{n}}{\cosh \left(\frac{1-i}{2} z\right)} &\stackrel{z \to \tilde{n} \pi(i-1)}{=} 2i\left\{ \begin{array}{c}
\frac{1}{1-i} \cdot \frac{1}{z-\tilde{n} \pi(i-1)} \\
+\sum_{k=1}^{\infty} (-1)^{k} \frac{\overline{\zeta}(2k)}{\pi^{2k}} \left( \frac{1-i}{2} \right)^{2k-1} (z-\tilde{n} \pi(i-1))^{2k-1}
\end{array} \right\}, \\
\frac{(-1)^{n}}{\sinh \left(\frac{1+i}{2} z\right)} &\stackrel{z \to n \pi(1+i)}{=} 2\left\{ \begin{array}{c}
\frac{1}{1+i} \cdot \frac{1}{z-n \pi(1+i)} \\
+\sum_{k=1}^{\infty} (-1)^{k} \frac{\overline{\zeta}(2k)}{\pi^{2k}} \left( \frac{1+i}{2} \right)^{2k-1} (z-n \pi(1+i))^{2k-1}
\end{array} \right\}, \\
\frac{(-1)^{n}}{\sinh \left(\frac{i-1}{2} z\right)} &\stackrel{z \to n \pi(1-i)}{=} 2\left\{ \begin{array}{c}
\frac{1}{i-1} \cdot \frac{1}{z-n \pi(1-i)} \\
+\sum_{k=1}^{\infty} (-1)^{k} \frac{\overline{\zeta}(2k)}{\pi^{2k}} \left( \frac{i-1}{2} \right)^{2k-1} (z-n \pi(1-i))^{2k-1}
\end{array} \right\},
\end{align}
where \(\tilde{n} := n - \frac{1}{2}\). 
\end{lem}

\begin{lem}\emph{(\cite{DCLMRT1992})}\label{lemBE}
	The Maclaurin series of \( \mathrm{cd}(u) \) and \( \mathrm{nd}(u) \) have the forms
	\begin{equation}
		\mathrm{cd}(u) = \sum_{n=0}^{\infty} S_{2n}(x) \frac{(-1)^n u^{2n}}{(2n)!} \quad \text{and} \quad \mathrm{nd}(u) = \sum_{n=0}^{\infty} A_{2n}(x) \frac{(-1)^n u^{2n}}{(2n)!},
	\end{equation}
	where $\mathrm{nd}(u) := 1 / \mathrm{dn}(u)$ and $S_{2n}(x), A_{2n}(x) \in \mathbb{Z}[x]$.
\end{lem}

\section{Berndt-Type Integrals via Hyperbolic Series}
In the present section, we delve into establishing rigorous and precise correspondences between Berndt-type integrals and Ramanujan-type hyperbolic series, leveraging the powerful tool of contour integration techniques. Our approach reveals deep connections between these apparently distinct mathematical objects.
\begin{thm}\label{thmCA}
 For any positive integer \( p \geq 5 \), the following identity holds:
\begin{align}
& \int_{0}^{\infty} \frac{x^{p} \, \mathrm{d}x}{[\cosh (2 x)-\cos (2 x)][\cosh x-\cos x]}-i^{p+1} \int_{0}^{\infty} \frac{x^{p} \, \mathrm{d}x}{[\cosh (2 x)-\cos (2 x)][\cosh x-\cos x]}\nonumber \\
& =\frac{1}{2^{p+3}} i (1+i)^{p+1} \pi^{p+1} \sum_{n=1}^{\infty} \frac{(-1)^{n-1}(2 n-1)^{p}}{\sinh^2 ((2 n-1) \pi / 2) \cosh ((2 n-1) \pi / 2)} \nonumber\\
& \quad-\frac{1}{4}((1+i))^{p-1} \pi^{p+1} \sum_{n=1}^{\infty} \frac{(-1)^{n} n^{p}}{\sinh (n \pi) \cosh ^{2}(n \pi)}-\frac{1}{2} (1+i)^{p-1} \pi^{p+1} \sum_{n=1}^{\infty} \frac{(-1)^{n}n^{p}}{\sinh ^{3}(n \pi )}\nonumber \\
& \quad+\frac{1}{4} p((1+i))^{p-1} \pi^{p} \sum_{n=1}^{\infty} \frac{(-1)^{n}n^{p-1}}{\sinh ^{2}(n \pi ) \cosh (n \pi )}.
\end{align}
\end{thm}
\begin{proof}
 Let \( z = x + iy \), \( x, y \in \mathbb{R} \). Consider the contour integral
\begin{equation}
\mathscr{A} _p=\lim_{R\rightarrow \infty} \int_{C_R}{\frac{z^p\mathrm{d}z}{\left( \cosh 2z-\cos 2z \right) \left( \cosh z-\cos z \right)}}=\lim_{R\rightarrow \infty} \int_{C_R}{F\left( z \right) \mathrm{d}z},
\end{equation}
where $C_R$ denotes a quarter-circular contour consisting of three components: the interval $[0,R]$, the quarter-circle $\Gamma_R$ defined by $|z|=R$ with $0 \le \arg (z)\le \pi/2$, and the interval $[iR,0]$.
It is evident that poles arise when
\[
[\cosh(2z) - \cos(2z)][\cosh z - \cos z] = 16 \cos\left\{ \frac{1 + i}{2}z \right\} \cos\left\{ \frac{1 - i}{2}z \right\} \sin^2\left\{ \frac{1 + i}{2}z \right\} \sin^2\left\{ \frac{1 - i}{2}z \right\} = 0.
\]
The poles enclosed by \( C_R \) are located at $ z_m = m\pi(1 + i) $ (where $m \geq 1, |z_m| < R$)  and $ z_n = (2n - 1)\pi(1 + i)/2$  (where $n \geq 1, |z_n| < R$). By virtue of Lemma \ref{lemBD}, the residues  \( \text{Res}[F(z), z] \) at these poles are expressed as follows:
\begin{equation}
\operatorname{Res}[F(z), z_n] = \frac{(-1)^{n - 1}(1 + i)^{p+1}(2n-1)^{p} \pi^{p} }{2^{p+4} \sinh^2((2n-1)\pi / 2) \cosh((2n-1)\pi / 2)}
\end{equation}
and
\begin{align}
\operatorname{Res[}F(z),z_m]&=\frac{i}{4}(1+i)^{p-1}\pi ^p\frac{(-1)^mm^p}{\sinh ^3(m\pi)} +\frac{i}{8}(1+i)^{p-1}\pi ^p\frac{(-1)^mm^p}{\sinh (m\pi) \cosh ^2(m\pi)}
	\nonumber \\
	&\quad -\frac{i}{8}p(1+i)^{p-1}\pi ^{p-1}\frac{(-1)^mm^{p-1}}{\sinh ^2(m\pi) \cosh (m\pi)}.
\end{align}
As \( R \to \infty \), we get
\[
\int_{\Gamma_{R}} \frac{z^{p} \mathrm{d}z}{[\cosh(2z) - \cos(2z)][\cosh z - \cos z]} = o(1).
\]
Applying the residue theorem and taking the limit as \( R \to \infty \), we obtain
\[
\begin{aligned}
& 2\pi i \sum_{m=1}^{\infty} \text{Res}\left[F(z), z_{m}\right] + 2\pi i \sum_{n=1}^{\infty} \text{Res}\left[F(z), z_{n}\right] \\
& = \lim_{R \to \infty} \int_{C_{R}} \frac{z^{p} \mathrm{d}z}{[\cosh(2z) - \cos(2z)][\cosh z - \cos z]} \\
& = \int_{0}^{\infty} \frac{x^{p} \mathrm{d}x}{[\cosh(2x) - \cos(2x)][\cosh x - \cos x]} - i \int_{0}^{\infty} \frac{(ix)^{p} \mathrm{d}x}{[\cosh(2ix) - \cos(2ix)][\cosh(ix) - \cos(ix)]} \\
& = \int_{0}^{\infty} \frac{x^{p} \mathrm{d}x}{[\cosh(2x) - \cos(2x)][\cosh x - \cos x]} - i^{p+1} \int_{0}^{\infty} \frac{x^{p} \mathrm{d}x}{[\cosh(2x) - \cos(2x)][\cosh x - \cos x]},
\end{aligned}
\]
by synthesizing these findings, we arrive at the identity we aimed to prove.
\end{proof}
\begin{thm}\label{thmCB}
For \( a, b, \theta \in \mathbb{R} \) and \( |\theta| < 3b\pi \), \( a, b \neq 0 \), we have
\begin{align}
& b^{2}\pi \sum_{n=1}^{\infty} \frac{(-1)^{n} \sinh(n\theta/a)}{\sinh(bn\pi/a) \cosh^{2}(bn\pi/a)} + ab\pi \sum_{n=1}^{\infty} \frac{(-1)^{n} \sin(n\theta/b)}{\sinh(an\pi/b)} - a\theta \sum_{n=1}^{\infty} \frac{(-1)^{n} \cos((2n-1)\theta/(2b))}{\sinh((2n-1)a\pi/(2b))} \notag\\
& + a^{2}\pi \sum_{n=1}^{\infty} \frac{(-1)^{n} \sin((2n-1)\theta/(2b)) \cosh((2n-1)a\pi/(2b))}{\sinh^{2}((2n-1)a\pi/(2b))} + \frac{\theta b}{2} = 0,\label{joyCE}\\
& b^{2}\pi \sum_{n=1}^{\infty} \frac{(-1)^{n} \cosh((2n-1)\theta/(2a))}{\sinh((2n-1)b\pi/(2a)) \cosh^{2}((2n-1)b\pi/(2a))} + a\theta \sum_{n=1}^{\infty} \frac{(-1)^{n} \sin((2n-1)\theta/(2b))}{\cosh((2n-1)a\pi/(2b))}\notag \\
& + a^{2}\pi \sum_{n=1}^{\infty} \frac{(-1)^{n} \cos((2n-1)\theta/(2b)) \sinh((2n-1)a\pi/(2b))}{\cosh^{2}((2n-1)a\pi/(2b))} + ab\pi \sum_{n=1}^{\infty} \frac{(-1)^{n} \cos(n\theta/b)}{\cosh(an\pi/b)} \notag\\
& + \frac{ab\pi}{2} = 0,\label{joyCF}\\
& b^{2}\pi^{2} \sum_{n=1}^{\infty} \frac{(-1)^{n} \cosh(n\theta/a)}{\sinh^{2}(n b\pi/a) \cosh(n b\pi/a)} + a \pi \theta \sum_{n=1}^{\infty} \frac{(-1)^{n} \sin(n\theta/b)}{\sinh(an\pi/b)} \notag\\
& + a^{2} \pi^{2} \sum_{n=1}^{\infty} \frac{(-1)^{n} \cos(n\theta/b) \cosh(an\pi/b)}{\sinh^{2}(an\pi/b)} + ab\pi^{2} \sum_{n=1}^{\infty} \frac{(-1)^{n-1} \cos((2n-1)\theta/(2b))}{\sinh((2n-1)a\pi/(2b))} \notag\\
&+ \frac{a^{2}\pi^{2}-5 b^{2}\pi^{2}+3 \theta^{2}}{12} = 0.\label{joyCG}
\end{align}
\end{thm}

\begin{proof}
	The core of the proof lies in the analysis of the pole distribution of \( S_{i}(s, \theta, a, b) \) (\( i = 1, 2, 3 \)) and  calculating their residues.
These functions are meromorphic over the entire complex plane, with simple poles at specific points.

To begin with, the function \( S_{1}(s, \theta, a, b) \) possesses poles at \( s = \pm \frac{\pi}{a} \) (simple poles), \( s = \pm \frac{\pi i}{a} \) (simple poles), \( s = \pm \frac{(2n-1)i}{2b} \) (double poles) and \( s = 0 \) (simple poles), where \( n \in \mathbb{N} \).

As \( s = \pm \frac{n}{a} \) and upon application of Lemma \ref{lemBC}, the residue is

\begin{equation} \label{joyCH}
\text{Res}\left[S_{1}, s = \pm \frac{n}{a}\right] = \frac{\pi^{3}}{a} \frac{(-1)^{n} \sinh(n\theta/a)}{\sinh(bn\pi/a) \cosh^{2}(bn\pi/a)}.
\end{equation}
As \( s = \pm \frac{ni}{b} \) and the residue is
\begin{equation}
\text{Res}\left[S_{1}, s = \pm \frac{ni}{b}\right] = \frac{\pi^{3}}{b} \frac{(-1)^{n} \sin(n\theta/a)}{\sinh(an\pi/b)}.
\end{equation}

From Lemma \ref{lemBD}, for every  \( n \in \mathbb{Z} \) we can derive the asymptotic expansions of several reciprocal quadratic trigonometric and hyperbolic functions, which are presented as follows
\begin{align}\label{joyCAO}
\left( \frac{\pi}{\cosh(b\pi s)} \right)^{2} &= -\frac{1}{\left( bs - \frac{2n-1}{2}i \right)^{2}} + 2\zeta(2) - 6\zeta(4) \left( bs - \frac{2n-1}{2}i \right)^{2} \nonumber\\
&\quad + o\left( \left( bs - \frac{2n-1}{2}i \right)^{2} \right) \quad \text{as } s \to \frac{2n-1}{2b}i.
\end{align}
As \( s = \pm \frac{(2n-1)i}{2b} \) and the residue is
\begin{align}
\text{Res}\left[S_{1}, s = \pm \frac{(2n-1)i}{2b}\right] &= -\frac{\pi^{2}\theta}{b^{2}} \frac{(-1)^{n} \cos((2n-1)\theta/(2b))}{\sinh((2n-1)a\pi/(2b))}\nonumber \\
&\quad + \frac{a\pi^{2}}{b^{3}} \frac{(-1)^{n} \sin((2n-1)\theta/(2b)) \cosh((2n-1)a\pi/(2b))}{\sinh^{2}((2n-1)a\pi/(2b))}.
\end{align}
Moreover, by applying equations \eqref{BF} and \eqref{BG}, we can derive the following expansion
\begin{equation}
S_{1}(s, \theta; a, b) = \frac{\pi^{2}}{\cosh^{2}(b\pi s)} \left( \frac{\theta}{ab} \frac{1}{s} + \frac{\theta^{3}}{6ab} s - \left( \frac{b}{a} - \frac{a}{b} \right) \zeta(2)\theta s + o(1) \right), \quad s \to 0.
\end{equation}

\begin{flushleft}Accordingly, as \( s = 0 \) and the residue is\end{flushleft}
\begin{equation} \label{joyCAC}
\text{Res}\left[S_{1}, s = 0\right] = \frac{\theta\pi^{2}}{ab}.
\end{equation}

By Lemma 3.1, summing the five contributions \eqref{joyCH}-\eqref{joyCAC} yields the desired result \eqref{joyCE}.

Next, the function \( S_{2}(s, \theta; a, b) \) possesses poles at \( s = \pm \frac{2n-1}{2a} \) (simple poles), \( s = \pm \frac{ni}{b} \) (simple poles), \( s = \pm \frac{(2n-1)i}{2b} \) (double poles) and \( s = 0 \) (simple poles), where \( n \in \mathbb{N} \).
As \( s = \pm \frac{2n-1}{2a} \) and by applying Lemma \ref{lemBC}, the residue is
\begin{equation}\label{S2joyCAC}
\text{Res}\left[S_{2}, s = \pm \frac{2n-1}{2a}\right] = \frac{\pi^{3}}{a} \frac{(-1)^{n} \cosh((2n-1)\theta/(2a))}{\sinh((2n-1)b\pi/(2a)) \cosh^{2}((2n-1)b\pi/(2a))}.
\end{equation}
As \( s = \pm \frac{ni}{b} \) and the residue is
\begin{equation}
\text{Res}\left[S_{2}, s = \pm \frac{ni}{b}\right] = \frac{\pi^{3}}{b} \frac{(-1)^{n} \cos(n\theta/b)}{\cosh(an\pi/b)}.
\end{equation}
For the second-order pole at\( s = \pm \frac{(2n-1)i}{2b} \), we make use of the expansion given in \eqref{joyCAO} to derive the following result
\begin{align}\label{joyCAF}
\text{Res}\left[S_{2}, s = \pm \frac{(2n-1)i}{2b}\right] &= \frac{\pi^{2}\theta}{b^{2}} \frac{(-1)^{n} \sin((2n-1)\theta/(2b))}{\cosh((2n-1)a\pi/(2b))} \nonumber\\
&\quad + \frac{a\pi^{2}}{b^{3}} \frac{(-1)^{n} \cos((2n-1)\theta/(2b)) \sinh((2n-1)a\pi/(2b))}{\cosh^{2}((2n-1)a\pi/(2b))}.
\end{align}
For \( s = 0 \), by using the expansions given in \eqref{BG}, we can derive the following
\[
S_{2}(s, \theta; a, b) = \frac{\pi^{3} \cosh(\theta s)}{\cos(a\pi s) \cosh^{2}(b\pi s)} \left( \frac{1}{bs} - \zeta(2)bs + o(1) \right), \quad s \to 0,
\]
which gives the residue
\begin{equation}
\text{Res}\left[S_{2}, s = 0\right] = \frac{\pi^{3}}{b}.
\end{equation}

\begin{flushleft}In accordance with Lemma \ref{lemBA}, the sum of these four residue contributions \eqref{S2joyCAC} to \eqref{joyCAF} serves to establish the intended identity \eqref{joyCF}.\end{flushleft}

Finally, the function \( S_{3}(s, \theta; a, b) \) have poles at \( s = \pm \frac{n}{a} \) (simple poles), \( s = \pm \frac{(2n-1)i}{2b} \) (simple poles), \( s = \pm \frac{ni}{b} \) (double poles) and \( s = 0 \) (simple poles), where \( n \in \mathbb{N} \).
As \( s = \pm \frac{n}{a} \) and through application of Lemma \ref{lemBC}, we calculate the residue at these simple poles
\begin{equation} \label{joyCAH}
\text{Res}\left[S_{3}, s = \pm \frac{n}{a}\right] = \frac{\pi^{3}}{a} \frac{(-1)^{n} \cosh(n\theta/a)}{\sinh^{2}(nb\pi/a) \cosh(nb\pi/a)}.
\end{equation}
As \( s = \pm \frac{(2n-1)i}{2b} \), the simple poles contribute
\begin{equation}
\text{Res}\left[S_{3}, s = \pm \frac{(2n-1)i}{2b}\right] = \frac{\pi^{3}}{b} \frac{(-1)^{n-1} \cos((2n-1)\theta/(2b))}{\sinh((2n-1)a\pi/(2b))}.
\end{equation}
For the second-order pole at \( s = \pm \frac{ni}{b} \), we make use of the asymptotic expansion provided in Lemma \ref{lemBC}
\[
\left( \frac{\pi}{\sinh(b\pi s)} \right)^{2} = \frac{1}{(bs - ni)^{2}} - 2\zeta(2) + 6\zeta(4) (bs - ni)^{2} + o\left( (bs - ni)^{2} \right) \quad \text{as } s \to ni,
\]
resulting in
\begin{equation} \label{joyCBO}
\text{Res}\left[S_{3}, s = \pm \frac{ni}{b}\right] = \frac{\pi^{2}\theta}{b^{2}} \frac{(-1)^{n} \sin(n\theta/b)}{\sinh(an\pi/b)} + \frac{a\pi^{3}}{b^{2}} \frac{(-1)^{n} \cos(n\theta/b) \cosh(an\pi/b)}{\sinh^{2}(an\pi/b)}.
\end{equation}
As \( s = 0 \), the asymptotic behavior in the vicinity of the origin is expressed as
\[
S_{3}(s, \theta; a, b) = \frac{\pi}{ \cosh(b\pi s)} \left(\frac{1}{b^{2}s^{2}} - 2\zeta(2) + o(1) \right)\left(\frac{1}{as} + \zeta(2)as + o(as) \right), \quad s \to 0,
\]
from which we extract the residue
\begin{equation}
\text{Res}\left[Z_{3}, s = 0\right] = \frac{a^{2}\pi^{3}-5 b^{2}\pi^{3}+3 \theta^{2}\pi}{6ab^{2}}.
\end{equation}
By applying Lemma \ref{lemBA} to the sum of these four residue contributions \eqref{joyCAH} to \eqref{joyCBO}, the identity \eqref{joyCG} is established.

 This concludes the proof of Theorem \ref{thmCB}.
\end{proof}
\begin{thm}
Let \( p > 1 \) be an odd integer, we establish the following equations:
\begin{align}
 &\sum_{n=1}^{\infty} \frac{(-1)^{n} n^{p-1}}{\cosh(n y) \sinh^{2}(n y)}= -\frac{\pi^{p-1}}{y^{p}} (p-1)(-1)^{(p-3)/2} \sum_{n=1}^{\infty} \frac{(-1)^{n} n^{p-2}}{\sinh(\pi^{2} n / y)}-\frac{w_p}{2y^2} \nonumber\\
&\qquad\qquad\qquad\qquad\qquad\quad + \frac{\pi^{p}}{2^{p-1} y^{p}} (-1)^{(p-1)/2} \sum_{n=1}^{\infty} \frac{(-1)^{n} (2n-1)^{p-1}}{\sinh((2n-1)\pi^{2}/(2y))} \nonumber\\
&\qquad\qquad\qquad\qquad\qquad\quad - \frac{\pi^{p+1}}{ y^{p+1}} (-1)^{(p-1)/2} \sum_{n=1}^{\infty} \frac{(-1)^{n} n^{p-1} \cosh(n\pi^{2}/y)}{\sinh^{2}(n\pi^{2}/y)},\label{joyCBB}\\
 &\sum_{n=1}^{\infty} \frac{(-1)^{n} n^{p}}{\sinh(n y) \cosh^{2}(n y)} = -\frac{\pi^{p+1}}{y^{p+1}} (-1)^{(p-1)/2} \sum_{n=1}^{\infty} \frac{(-1)^{n} n^{p}}{\sinh(\pi^{2} n / y)} \nonumber\\
&\quad \qquad\qquad\qquad\qquad\qquad+ \frac{p\pi^{p}}{2^{p-1} y^{p+1}} (-1)^{(p-1)/2} \sum_{n=1}^{\infty} \frac{(-1)^{n} (2n-1)^{p-1}}{\sinh((2n-1)\pi^{2}/(2y))} \nonumber\\
&\quad\qquad\qquad\qquad\qquad\qquad - \frac{\pi^{p+2}}{2^{p} y^{p+2}} (-1)^{(p-1)/2} \sum_{n=1}^{\infty} \frac{(-1)^{n} (2n-1)^{p} \cosh((2n-1)\pi^{2}/(2y))}{\sinh^{2}((2n-1)\pi^{2}/(2y))},\label{joyCBC}\\
 &\sum_{n=1}^{\infty} \frac{(-1)^{n} (2n-1)^{p}}{\sinh^{2}((2n-1)y/2) \cosh((2n-1)y/2)}= -\frac{\pi^{p+1}}{y^{p+1}} (-1)^{(p-1)/2} \sum_{n=1}^{\infty} \frac{(-1)^{n} (2n-1)^{p}}{\cosh((2n-1)\pi^{2}/(2y))} \nonumber\\
&\quad \qquad\qquad\qquad\qquad\qquad\qquad\qquad\qquad\qquad- \frac{p 2^{p} \pi^{p}}{y^{p+1}} (-1)^{(p-1)/2} \sum_{n=1}^{\infty} \frac{(-1)^{n} n^{p-1}}{\cosh(\pi^{2} n / y)}\nonumber \\
&\quad\qquad\qquad\qquad\qquad\qquad\qquad\qquad\qquad\qquad + \frac{2^{p} \pi^{p+2}}{y^{p+2}} (-1)^{(p-1)/2} \sum_{n=1}^{\infty} \frac{(-1)^{n} n^{p} \sinh(\pi^{2} n / y)}{\cosh^{2}(\pi^{2} n / y)},\label{joyCBD}
\end{align}
where \( w_{3} = 1 \) and \( w_{p} = 0 \) if \( p \geq 5 \).
\end{thm}

\begin{proof}
We present a detailed proof of the first identity, as the proofs for the other two proceed analogously. By computing the \( (p-1) \)-th odd derivative of \eqref{joyCG} in Theorem \ref{thmCB} with respect to \( \theta \), then \( \theta \to 0 \), we arrive at the following result
\[
\begin{aligned}
\left. \frac{d^{p-1}}{d\theta^{p-1}} \left[ \cosh\left( \frac{n\theta}{a} \right) \right] \right|_{\theta=0} &= \left( \frac{n}{a} \right)^{p-1}, \\
\left. \frac{d^{p-1}}{d\theta^{p-1}} \left[ \theta \sin\left( \frac{n\theta}{b} \right) \right] \right|_{\theta=0} &= (p-1) (-1)^{(p-3)/2} \left( \frac{n}{b} \right)^{p-2}, \\
\left. \frac{d^{p-1}}{d\theta^{p-1}} \left[ \cos\left( \frac{(2n-1)\theta}{2b} \right) \right] \right|_{\theta=0} &= (-1)^{(p-1)/2} \left( \frac{(2n-1)}{2b} \right)^{p-1}, \\
\left. \frac{d^{p-1}}{d\theta^{p-1}} \left[ \cos\left( \frac{n\theta}{b} \right) \right] \right|_{\theta=0} &= (-1)^{(p-1)/2} \left( \frac{n}{b} \right)^{p-1}.
\end{aligned}
\]
By utilizing the above derivatives and substituting \( y = \frac{b\pi}{a} \) into \eqref{joyCG}, we can readily obtain the following result
\begin{align}
& ab^{2}\delta_{p} - (-1)^{(p-1)/2} \frac{a\pi^{3}}{2^{p-1} b^{p-2}} \sum_{n=1}^{\infty} \frac{(-1)^{n} (2n-1)^{p-1}}{\sinh((2n-1)\pi^{2}/(2y))}  + \frac{b^{2} \pi^{3}}{a^{p-1}} \sum_{n=1}^{\infty} \frac{(-1)^{n} n^{p-1}}{\sinh^{2}( n y) \cosh( n y)}\nonumber \\
& \quad + (-1)^{(p-3)/2} \frac{a\pi^{2}(p-1)}{b^{p-2}} \sum_{n=1}^{\infty} \frac{(-1)^{n} n^{p}}{\sinh(\pi^{2} n / y)}  + (-1)^{(p-1)/2} \frac{a^{2} \pi^{3}}{2^{p} b^{p-1}} \sum_{n=1}^{\infty} \frac{(-1)^{n} n^{p-1} \cosh(n\pi^{2}/y)}{\sinh^{2}(n\pi^{2}/y)} \nonumber\\
&= 0,
\end{align}
where \( \delta_{3} = \frac{\pi}{2ab^{2}} \) and \( \delta_{p} = 0 \) if \( p \geq 5 \). After rearrangement, we can obtain the equation \eqref{joyCBB}.

The proofs of equations \eqref{joyCBC} and \eqref{joyCBD} proceed analogously, with the following key distinctions: for \eqref{joyCBC}, we compute the $(p-1)$-th order derivative (i.e., an even-order derivative) of equation \eqref{joyCE} with respect to \(\theta\); for \eqref{joyCBD}, we compute the $p$-th order odd derivative of equation \eqref{joyCF} with respect to \(\theta\).
\end{proof}

\section{Evaluations of Hyperbolic Summations via Jacobi Functions}

We now proceed to evaluate the following reciprocal hyperbolic series of the Ramanujan type:
\begin{align}
C_{p,2}(y)& := \sum_{n=1}^{\infty} \frac{(-1)^{n} n^{p}}{\sinh(ny) \cosh^{2}(ny)},
\\
C'_{p,2}(y) &:= \sum_{n=1}^{\infty} \frac{(-1)^{n} n^{p-1}}{\sinh^{2}(ny) \cosh(ny)},
\\
\overline{C}_{p,2}(y) &:= \sum_{n=1}^{\infty} \frac{(-1)^{n} (2n-1)^{p}}{\sinh^{2}((2n-1)y/2) \cosh((2n-1)y/2)}.
\end{align}

By applying the Fourier series expansions and Maclaurin series expansions of relevant Jacobi elliptic functions, let $x,y$ and $z$ satisfy the relations in \eqref{eveBB}.
\begin{thm}\label{thmDA}
Let \( p \geq 5 \) be an odd integer. We have
\begin{align}
C_{p,2}(y) &= (-1)^{(p-1)/2} \frac{(p-1)! 2^{p+1} z(1-x)}{y^{p+1}} R_{p-1}(1-x) - \frac{p 2^{p+1} z(1-x)}{2^{p}} S_{p-1}(1-x) \sqrt{1-x} \nonumber\\
&\quad - \frac{z^{2} 2^{p+2} (1-x)}{2^{2p}} \frac{d^{p}}{dx^{p}} \left[ S_{p-1}(1-x) \sqrt{1-x} \right],
\end{align}
where \( R_{p-1}(1-x) = \frac{(-x)^{(p-3)/2}}{(p-1)!} q_{p-1}\left( \frac{1-x}{x} \right) \in \mathbb{Q}[x] \) and \( q_{p-1}(x) \) represents coefficients in the Maclaurin expansion of \( \text{sn}(u)^{2} \); \( S_{p-1}(x) \) represents coefficients in the Maclaurin expansion of \( \text{cd}(u) \).
\end{thm}

\begin{proof}Beginning with equation \eqref{joyCBC}, we derive
\begin{align}
C_{p,2}(y) &= -(-1)^{(p-1)/2} \frac{p \pi^{p+1}}{y^{p+1}} X_{p,1}\left( \frac{\pi^{2}}{y} \right) + (-1)^{(p-1)/2} \frac{p \pi^{p}}{2^{p-1} y^{p+1}} T_{p,1}\left( \frac{\pi^{2}}{y} \right) \nonumber\\
&\quad - (-1)^{(p-1)/2} \frac{\pi^{p+2}}{2^{p} y^{p+2}} DT_{p,2}\left( \frac{\pi^{2}}{y} \right).
\end{align}
We now turn to computing the series \(T_{p,1}(y)\). By Lemma \ref{lemBE}, we can derive the Maclaurin expansion of \(\text{cd}(u)\)
\begin{equation} \label{skyDF}
\text{cd}(u) = \sum_{n=0}^{\infty} S_{2n}(x) \frac{(-1)^{n} u^{2n}}{(2n)!}.
\end{equation}
Drawing on known results from \cite{DCLMRT1992}, we present the Fourier series expansions of the Jacobi elliptic function \(\text{cd}(u)\)
\begin{equation}
\text{cd}(u) = \frac{2\pi}{Kk} \sum_{n=0}^{\infty} \frac{(-1)^{n} q^{n+1/2}}{1 - q^{2n+1}} \cos\left( (2n+1) \frac{\pi u}{2K} \right) \quad (|q| < 1).
\end{equation}
Additionally, on page 165 of Ramanujan's Notebooks (III)\cite{Berndt2016}, we find that
\begin{equation}
\frac{1}{2} \sqrt{\pi} \text{cd}(u) = \sum_{n=0}^{\infty} \frac{(-1)^{n} \cos((2n+1)u)}{\sinh((2n+1)y/2)}.
\end{equation}
Applying \( q = q(x) = e^{-y} \), then
\begin{align}\label{skyDI}
\text{cd}(u) &= \frac{2\pi}{Kk} \sum_{n=0}^{\infty} \frac{(-1)^{n} q^{n+\frac{1}{2}} (2n+1)^{2}}{1 - q^{2n+1}} \sum_{j=0}^{\infty} \frac{(-1)^{j} \left( \frac{\pi u}{2K} \right)^{2j}}{(2j)!} \nonumber\\
&= \frac{\pi}{Kk} \sum_{j=0}^{\infty} \frac{(-1)^{j} \left( \frac{\pi u}{2K} \right)^{2j}}{(2j)!} Y_{2j+1,1}(y).
\end{align}
By comparing the coefficients of \(u^{2n}\) in \eqref{skyDF} and \eqref{skyDI}, we can deduce that
\begin{equation} \label{skyDAO1}
T_{2n+1,1}(y) = \sum_{j=1}^{\infty} \frac{(-1)^{j} (2j-1)^{2n}}{\sinh\left( \frac{2j-1}{2}y \right)} = (-1)^{n} z^{2n+1} \sqrt{x} \frac{S_{2n+1}(x)}{2}.
\end{equation}
Differentiating equation \eqref{skyDAO1} yields the following relation
\[
\frac{d}{dy} T_{p,1}(y) = \sum_{n=1}^{\infty} \frac{d}{dy} \frac{(-1)^{n} (2n-1)^{p-1}}{\sinh\left( \frac{2n-1}{2}y \right)} = -\frac{1}{2} DT_{p,2}(y).
\]
Using \eqref{eveBD}, we obtain
\begin{equation} \label{skyDAA}
DT_{p,2}(y) = 2x(1-x) z^{2} \frac{d}{dx} T_{p,1}(y).
\end{equation}
Applying Lemma \ref{lemBB}, for \( \Omega(x, e^{-\pi z}, y, z') = 0 \), we have
\[
\Omega\left( 1 - x, e^{-\pi z/2}, y/z, \frac{1}{\pi} \left( \frac{1}{x(1-x) z} - y z' \right) \right) = 0.
\]
Hence, equations \eqref{skyDAO1} and \eqref{skyDAA} can be rewritten as follows
\begin{equation} \label{skyDAB}
T_{p,1}\left( \frac{\pi^{2}}{y} \right) = -(-1)^{(p-1)/2} \left( \frac{yz}{\pi} \right)^{p} \frac{S_{p-1}(1-x)}{2} \sqrt{1-x},
\end{equation}
and
\begin{align}\label{skyDAC}
DT_{p,2}\left( \frac{\pi^{2}}{y} \right) &= 2x(1-x) \left( \frac{yz}{\pi} \right)^{2} \frac{d}{d(1-x)} T_{p,1}\left( \frac{\pi^{2}}{y} \right) \nonumber\\
&= -2x(1-x) \left( \frac{yz}{\pi} \right)^{2} \frac{d}{dx} T_{p,1}\left( \frac{\pi^{2}}{y} \right).
\end{align}
Furthermore, as is known from Rui-Xu-Zhao's paper \cite[Thm. 3.15]{RXZ2023}, and by applying Lemma \ref{lemBB}, for \(\Omega(x, e^{-\pi z}, y, z') = 0\), we obtain that
\begin{equation} \label{skyDAD}
\sum_{n=1}^{\infty} \frac{(-1)^{n} n^{p}}{\sinh\left( \frac{\pi^{2} n}{y} \right)} = -\frac{(p-1)!}{2^{p+1}} \left( \frac{yz}{\pi} \right)^{2} x(1-x) R_{p-1}(1-x),
\end{equation}
where \( R_{p-1}(1-x) = \frac{(-x)^{(p-3)/2}}{(p-1)!} q_{p-1}\left( \frac{1-x}{-x} \right) \in \mathbb{Q}[x] \).
At last, by combining equations \eqref{joyCBC}, \eqref{skyDAO1}, \eqref{skyDAB}, \eqref{skyDAC} and \eqref{skyDAD}, we thus complete the proof of Theorem \ref{thmDA}
\end{proof}

\begin{exa} By \emph{Mathematica}, we have
	\begin{align}
		\sum_{n=1}^{\infty} \frac{(-1)^n n^5}{\sinh(ny) \cosh^2(ny)}
		&= \frac{1}{64} x z^6 \sqrt{1-x}\left\{ -10 z'(-4x + 5x^2)(1 - x) + 8\sqrt{1-x}(2x - 1) \right.\notag \nonumber\\
		&\quad \left.- z(-25x^2 + 32x - 8) \right\},\nonumber \\
		\sum_{n=1}^{\infty} \frac{(-1)^n n^9}{\sinh(ny) \cosh^2(ny)}\nonumber
		&= \frac{1}{1024} x z^{10} \sqrt{1 - x}\left\{ z(12465x^4 - 28048x^3 + 20064x^2 - 4608x + 128)\right. \notag \nonumber\\
		&\quad - 18 z'(1 - x)(1385x^4 - 2424x^3 + 1104x^2 - 64x) \notag\nonumber \\
		&\quad \left.+ 128x\sqrt{1-x}(-62x^3 + 93x^2 - 33x + 1) \right\},\nonumber \\
		\sum_{n=1}^{\infty} \frac{(-1)^n n^{13}}{\sinh(ny) \cosh^2(ny)}\nonumber
		&= \frac{1}{8192} x z^{14}\sqrt{1-x} \left\{ - 13(1 - x) z' \begin{pmatrix} 2702765x^6 - 7432604x^5 + 7052528x^4 \\ -2586112x^3 + 264448x^2 - 1024x  \end{pmatrix}\right. \notag\nonumber \\
		&\quad \left.- \frac{1}{2}z \left( \begin{array}{c}-35135945x^{6} + 114191824x^{5} - 137798792x^{4} \\ + 74523008x^{3} -16838912x^{2} +1060864x - 2048 \\ \end{array} \right) \right\}.\nonumber
	\end{align}
\end{exa}

Setting $x=1/2$ and $p=4m-3$ in Theorem \ref{thmDA} yields the following corollary.
\begin{cor}
Let \( \Gamma = \Gamma(1/4) \). Then, for any integer \( m > 1 \), we have
\begin{align}\label{skyAE}
\sum_{n=1}^{\infty} \frac{(-1)^{n} n^{4m-3}}{\sinh(n\pi) \cosh^{2}(n\pi)} &= -\frac{1}{2^{8m-5}} (4m-3) S_{4m-4}\left( \frac{1}{2} \right) \frac{1}{\sqrt{2}} \frac{\Gamma^{8m-6}}{\pi^{(12m-7)/2}} \nonumber\\
&\quad -\frac{1}{2^{8m-2}} \left( S_{4m-4}'\left( \frac{1}{2} \right) - S_{4m-4}\left( \frac{1}{2} \right) \right) \frac{1}{\sqrt{2}} \frac{\Gamma^{8m-2}}{\pi^{(12m-3)/2}}.
\end{align}
\end{cor}

\begin{exa}
Set \( \Gamma := \Gamma(1/4) \). By \emph{Mathematica}, we have
\begin{align*}
\sum_{n=1}^{\infty} \frac{(-1)^{n} n^{5}}{\sinh(n\pi) \cosh^{2}(n\pi)} &= \frac{15\Gamma^{10}}{8192\sqrt{2}\pi^{17/2}} - \frac{7\Gamma^{14}}{65536\sqrt{2}\pi^{21/2}} ,\\
\sum_{n=1}^{\infty} \frac{(-1)^{n} n^{9}}{\sinh(n\pi) \cosh^{2}(n\pi)} &= -\frac{3969\Gamma^{18}}{8388608\sqrt{2}\pi^{29/2}} + \frac{1809\Gamma^{22}}{67108864\sqrt{2}\pi^{33/2}} ,\\
\sum_{n=1}^{\infty} \frac{(-1)^{n} n^{13}}{\sinh(n\pi) \cosh^{2}(n\pi)} &= \frac{5756751\Gamma^{26}}{8589934592\sqrt{2}\pi^{41/2}} - \frac{2630583\Gamma^{30}}{68719476736\sqrt{2}\pi^{45/2}}.
\end{align*}
\end{exa}

\begin{thm}\label{thmDE}
Let \( p \geq 5\) be an odd integer. We have
\begin{align}
C'_{p-1,2}(y) &=(-1)^{(p-1)/2} \frac{(p-1)(p-3)!(x-1)^{2}x^{2}z^{p}z'}{2^{p-1}} R_{p-3}(1-x)\nonumber\\
&\quad + \frac{(-1)^{p}\sqrt{1-x}z^{p}}{2^{p}} S_{p-1}(1-x) \sqrt{1-x} \nonumber\\
&\quad + \frac{(-1)^{(p-1)/2}(p-3)! x (1-x)z^{p+1}}{2^{p-1}} \frac{d^{p}}{dx^{p}} \left[ (x-1)x R_{p-3}(1-x) \right],
\end{align}
where \( R_{p-3}(1-x) = \frac{(-x)^{(p-5)/2}}{(p-3)!} q_{p-3}\left( \frac{1-x}{x} \right) \in \mathbb{Q}[x] \) and \( q_{p-3}(x) \) represents coefficients in the Maclaurin expansion of \( \text{sn}^{2}(u) \); \( S_{p-1}(1-x) \) represents coefficients in the Maclaurin expansion of \( \text{cd}(u) \).
\end{thm}

\begin{proof} Starting from equation \eqref{joyCBB}, we derive
\begin{align}\label{equ-CDp-12y}
C'_{p-1,2}(y) &= -(-1)^{(p-3)/2}(p-1)\frac{\pi^{p-1}}{y^{p}} X_{p-2,1}\left( \frac{\pi^{2}}{y} \right) + (-1)^{(p-1)/2} \frac{p \pi^{p}}{2^{p-1} y^{p}} T_{p,1}\left( \frac{\pi^{2}}{y} \right)\nonumber \\
&\quad - (-1)^{(p-1)/2} \frac{\pi^{p+1}}{ y^{p+1}} DX_{p-1,2}\left( \frac{\pi^{2}}{y} \right).
\end{align}
The expression for $T_{p,1}(\pi^2/y)$ can be obtained from \eqref{skyDAB}.
Furthermore, as is known from Rui-Xu-Zhao's paper \cite[Thm. 3.15]{RXZ2023}, and by applying Lemma \ref{lemBB}, we obtain
\begin{equation}
\sum_{n=1}^{\infty} \frac{(-1)^{n} n^{p}}{\sinh\left( \frac{\pi^{2} n}{y} \right)} = -\frac{(p-1)!}{2^{p+1}} \left( \frac{yz}{\pi} \right)^{2} x(1-x) R_{p-1}(1-x),
\end{equation}
where \( R_{p-1}(1-x) = \frac{(-x)^{(p-3)/2}}{(p-1)!} q_{p-1}\left( \frac{x-1}{x} \right) \in \mathbb{Q}[x] \). Hence,
\begin{align}\label{skyDAI}
X_{p-2,1}\left( \frac{\pi^{2}}{y} \right) &= -\frac{(p-3)!}{2^{p-1}} \left( \frac{yz}{\pi} \right)^{p-1} x(1-x) R_{p-3}(1-x),
\end{align}
\begin{align}\label{skyDAO}
DX_{p-1,2}\left( \frac{\pi^{2}}{y} \right) &= x(1-x) \left( \frac{yz}{\pi} \right)^{2} \frac{d}{d(1-x)} X_{p-2,1}\left( \frac{\pi^{2}}{y} \right) \nonumber\\
&= -x(1-x) \left( \frac{yz}{\pi} \right)^{2} \frac{d}{dx} X_{p-2,1}\left( \frac{\pi^{2}}{y} \right).
\end{align}
Finally, substituting \eqref{skyDAB}, \eqref{skyDAI} and \eqref{skyDAO} into \eqref{equ-CDp-12y} and performing direct calculations completes the proof.
\end{proof}

\begin{exa}
By \emph{Mathematica}, we have
\begin{align*}
\sum_{n=1}^{\infty} \frac{(-1)^{n}n^{4}}{\sinh^{2}(ny) \cosh(ny)}
& =-\frac{1}{32} x z^{5} \left\{(-4+5x)\sqrt{1-x} + 16z'(1-x)^{2} x\right.\notag \nonumber\\
&\quad \left.+ 4z(1-2x)(1-x) \right\},\\
\sum_{n=1}^{\infty} \frac{(-1)^{n}n^{8}}{\sinh^{2}(ny) \cosh(ny)}
&= -\frac{1}{512}x z^{9} \left\{(-64+1104x-2424x^{2}+1385x^{3})\sqrt{1-x} \right.\nonumber\\
&\quad+ 256z'(1 - x)^{2} x(2-17x + 17x^{2})\\
&\quad \left.+ 32z(x - 1)( -2 + 38x + 68x^{3} - 102x^{2}) \right\},\\
\sum_{n=1}^{\infty} \frac{(-1)^{n}n^{12}}{\sinh^{2}(ny) \cosh(ny)} &=-\frac{1}{8192}x z^{13} \left\{ \sqrt{1-x} \left( \begin{array}{c} -1024 + 264448 x - 2586112 x^{2} \\ 7052528x^{3} - 7432604x^{4} + 2702765x^{5} \end{array} \right) \right.\\
&\quad +256 z \left( 2-524x + 6222 x^{2} -23320x^{3} + 38350x^{4} - 29022x^{5} + 8292x^{6} \right) \\
&\quad \left.+6144 z'\left( 2x-263x^{2} + 2161 x^{3} - 6305x^{4} + 8551x^{5} - 5528x^{6} + 1382x^{7} \right)\right\}.
\end{align*}
\end{exa}

Setting $x=1/2$ and $p=4m-3$ in Theorem \ref{thmDE} gives the following corollary.
\begin{cor}
Let \( \Gamma = \Gamma(1/4) \). For any integer \( m > 1 \), we have
\begin{align}\label{skyDBA}
\sum_{n=1}^{\infty} \frac{(-1)^{n} n^{4m-4}}{\sinh^{2}(n\pi) \cosh(n\pi)} &= -\frac{1}{2^{8m-6}} S_{4m-4}\left( \frac{1}{2} \right) \frac{1}{\sqrt{2}} \frac{\Gamma^{8m-6}}{\pi^{(12m-9)/2}}\nonumber \\
&\quad -\frac{1}{2^{8m-5}} (4m-4)(4m-6)! R_{4m-6}\left( \frac{1}{2} \right)  \frac{\Gamma^{8m-8}}{\pi^{6m-5}}.
\end{align}
\end{cor}

\begin{exa}
Set \( \Gamma = \Gamma(1/4) \). By \emph{Mathematica}, we have
\begin{align*}
\sum_{n=1}^{\infty} \frac{(-1)^{n}n^{4}}{\sinh^{2}(n\pi) \cosh(n\pi)} &= -\frac{\Gamma^{8}}{256\pi^{7}} + \frac{3\Gamma^{10}}{4096\sqrt{2}\pi^{15/2}},\\
\sum_{n=1}^{\infty} \frac{(-1)^{n}n^{8}}{\sinh^{2}(n\pi) \cosh(n\pi)}& = \frac{9\Gamma^{16}}{16384\pi^{13}} - \frac{441\Gamma^{18}}{4194304\sqrt{2}\pi^{27/2}},\\
\sum_{n=1}^{\infty} \frac{(-1)^{n}n^{12}}{\sinh^{2}(n\pi) \cosh(n\pi)} &= -\frac{567\Gamma^{24}}{1048576\pi^{19}} + \frac{442827\Gamma^{26}}{4294967296\sqrt{2}\pi^{39/2}}.
\end{align*}
\end{exa}

\begin{thm}\label{thmDI}
Let \( p \geq 5 \) be an odd integer. We have
\begin{align}
\overline{C}_{p,2}(y) &= \frac{1}{2^{p+1}} P_{p-1}(1-x) \sqrt{x(1-x)} + p 2^{p+1} z x(1-x) A_{p-1}(1-x) \sqrt{x} \nonumber\\
&\quad + x(1-x) z^{p+2} \frac{d}{dx} \left[ A_{p-1}(1-x) \sqrt{x} \right],
\end{align}
where \( P_{p}(1-x) \) denotes coefficients in the Maclaurin expansion of \( \text{sd}(u) \) and \( A_{p-1}(1-x) \) denotes coefficients in the Maclaurin expansion of \( \text{nd}(u) \).
\end{thm}

\begin{proof} We begin with equation \eqref{joyCBD}, which gives
\begin{align}\label{skyDBC}
\overline{C}_{p,2}(y) &= -(-1)^{(p-1)/2} \frac{2^{p} \pi^{p+1}}{y^{p+1}} X'_{p,1}\left( \frac{\pi^{2}}{y} \right) - (-1)^{(p-1)/2} \frac{2^{p} \pi^{p}}{y^{p}} B_{p,1}\left( \frac{\pi^{2}}{y} \right) \nonumber\\
&\quad + (-1)^{(p-1)/2} \frac{2^{p} \pi^{p+2}}{y^{p+2}} DB_{p,2}\left( \frac{\pi^{2}}{y} \right).
\end{align}

We now proceed to compute the series \(B_{p,1}(y)\). From Lemma \ref{lemBE} and \cite{DCLMRT1992}, we can derive the Maclaurin expansion of \(\text{nd}(u)\)
\begin{equation} \label{skyDBD}
\text{nd}(u) = \sum_{n=0}^{\infty} A_{2n}(x) \frac{(-1)^{n} u^{2n}}{(2n)!}\quad\text{and} \quad \frac{2q^{n}}{1 + q^{2n}} = \frac{1}{\cosh(\pi n / y)}.
\end{equation}
And its Fourier series representation
\begin{equation}
\text{nd}(u) = \frac{\pi}{2Kk^{\prime}} + \frac{2\pi}{Kk^{\prime}} \sum_{n=1}^{\infty} \frac{(-1)^{n} q^{n}}{1 - q^{2n+1}} \cos\left( 2n \frac{\pi u}{2K} \right) \quad (|q| < 1).
\end{equation}
Setting \( q = q(x) = e^{-y} \), we expand
\begin{align} \label{skyDBF}
\text{nd}(u) &= \frac{\pi}{2Kk^{\prime}} + \frac{2\pi}{Kk^{\prime}} \sum_{n=1}^{\infty} \frac{(-1)^{n} q^{n} n^{2j}}{1 + q^{2n}} \sum_{j=0}^{\infty} \frac{(-1)^{j} \left( \frac{\pi u}{K} \right)^{2j}}{(2j)!} \nonumber\\
&= \frac{\pi}{2Kk^{\prime}} + \frac{\pi}{Kk^{\prime}} \sum_{j=0}^{\infty} \frac{(-1)^{j} \left( \frac{\pi u}{K} \right)^{2j}}{(2j)!} B_{2j+1,1}(y).
\end{align}
By comparing the coefficients of \(u^{2n}\) in \eqref{skyDBD} and \eqref{skyDBF}, we deduce
\begin{equation}
B_{p,1}(y) = (-1)^{(p-1)/2} 2^{p} z \sqrt{1-x} \frac{A_{p-1}(x)}{2^{p}}.
\end{equation}
By applying Lemma \ref{lemBB}, we can find that
\begin{equation} \label{skyDBH}
B_{p,1}\left( \frac{\pi^{2}}{y} \right) = (-1)^{(p-1)/2} \left( \frac{yz}{\pi} \right)^{p} \sqrt{x} \frac{A_{p-1}(1-x)}{2^{p}}.
\end{equation}
By utilizing the results from Rui-Xu-Zhao's paper \cite[Thm. 3.3]{RXZ2023}, we obtain
\begin{equation}
X'_{p,1}(y) = \sum_{n=1}^{\infty} \frac{(-1)^{n} (2n-1)^{p}}{\cosh\left( \frac{(2n-1)y}{2} \right)} = -(-1)^{(p-1)/2} 2^{p+1} \sqrt{x(1-x)} \frac{P_{p}(x)}{2}.
\end{equation}
The derivative of \(B'_{p,1}(y)\) satisfies the following relation
\[
\frac{d}{dy} B_{p,1}(y) = \sum_{n=1}^{\infty} \frac{d}{dy} \frac{(-1)^{n} n^{p-1}}{\cosh(\pi n / y)} = -DB_{p,2}(y).
\]
From the fundamental relation \eqref{eveBD}, we have
\begin{equation} \label{skyDCO}
DB_{p,2}(y) = x(1-x) z^{2} \frac{d}{dx} B_{p,1}(y).
\end{equation}
Using Lemma \ref{lemBB} and \eqref{skyDBH}-\eqref{skyDCO}, we get
\begin{equation} \label{skyDCA}
X'_{p,1}\left( \frac{\pi^{2}}{y} \right) = -(-1)^{(p-1)/2} \left( \frac{yz}{\pi} \right)^{p+1} \frac{P_{p}(1-x)}{2} \sqrt{x(1-x)}
\end{equation}
and
\begin{align}\label{skyDCB}
DB_{p,2}\left( \frac{\pi^{2}}{y} \right) &= x(1-x) \left( \frac{yz}{\pi} \right)^{2} \frac{d}{d(1-x)} B_{p,1}\left( \frac{\pi^{2}}{y} \right) \nonumber\\
&= -x(1-x) \left( \frac{yz}{\pi} \right)^{2} \frac{d}{dx} B_{p,1}\left( \frac{\pi^{2}}{y} \right).
\end{align}
Finally, substituting \eqref{skyDBH}, \eqref{skyDCA} and \eqref{skyDCB} into \eqref{skyDBC} and carrying out the necessary calculations completes the proof.
\end{proof}

\begin{exa} By \emph{Mathematica}, we have
\begin{align*}
 \sum_{n=1}^{\infty} \frac{(-1)^{n}(2n-1)^{5}}{\sinh^{2}((2n-1)y/2) \cosh((2n-1)y/2)} &= \frac{1}{2} z^{6} \sqrt{x-x^2}\left\{
\begin{array}{l}
16x^2 - 6x + 1-10z'x\sqrt{1-x}(5x^2 -6x+1) \\
-z\sqrt{1-x}(25x^2-8x+1)
\end{array}
\right\},\\
\sum_{n=1}^{\infty} \frac{(-1)^{n}(2n-1)^{9}}{\sinh^{2}((2n-1)y/2) \cosh((2n-1)y/2)}
& = \frac{1}{2} z^{10}\sqrt{x-x^2} \left\{
\begin{array}{l}
	7936x^4 - 15872x^3
+ 9168x^2 - 1232x + 1 \\
-18z'x\sqrt{1-x}\left(\begin{array}{l}
1385x^4 -3116x^3 \\
+2142x^2 -412x+1
\end{array}\right)\\
-z\sqrt{1-x}\left(\begin{array}{l}
12465x^4 -21812x^3\\
+10710x^2 -1236x+1
\end{array}\right)
\end{array}
\right\},\\
\sum_{n=1}^{\infty}\frac{(-1)^{n}(2n-1)^{13}}{\sinh^{2}((2n-1)y/2) \cosh((2n-1)y/2)}
& = \frac{1}{2} z^{14}\sqrt{x-x^2} \left\{
\begin{array}{l}
\left(\begin{array}{l}
	22368256x^6 - 67104768x^5
	 + 71997696x^4\\-32154112x^3
	 +4992576x^2-99648x+1
\end{array}\right) \\
-26z'x\sqrt{1-x}\left(\begin{array}{l}
	2702765x^6-8783986x^5\\
	+10430983x^4-5353260x^3\\
	+1036715x^2 -33218x +1
\end{array}\right)\\
-z\sqrt{1-x}\left(\begin{array}{l}
	35135945x^6-96623846x^5\\
	+93878847x^4-37472820x^3\\
	+5183575x^2 -99654x +1
\end{array}\right)\\
\end{array}
\right\}.
\end{align*}
\end{exa}

Setting $x=1/2$ and $p=4m-3$ in Theorem \ref{thmDI}, a straightforward computation yields the following corollary.
\begin{cor}Let \( \Gamma = \Gamma(1/4) \). For any integer \( m > 1\), we have
\begin{align} \label{skyDCC}
& \sum_{n=1}^{\infty} \frac{(-1)^{n}(2n-1)^{4m-3}}{\sinh^{2}((2n-1)\pi/2) \cosh((2n-1)\pi/2)} \nonumber\\
& = -\frac{1}{2^{4m-2}} (4m-3) A_{4m-4}\left( \frac{1}{2} \right) \frac{1}{\sqrt{2}} \frac{\Gamma^{8m-6}}{\pi^{(12m-7)/2}}+\frac{1}{2^{4m}}P_{4m-3}\left( \frac{1}{2} \right) \frac{\Gamma^{8m-4}}{\pi^{6m-3}} \nonumber\\
& \quad -\frac{1}{2^{4m+1}} (4m-4) \left( A_{4m-4}'\left( \frac{1}{2} \right) + A_{4m-4}\left( \frac{1}{2} \right) \right) \frac{1}{\sqrt{2}} \frac{\Gamma^{8m-2}}{\pi^{(12m-3)/2}}.
\end{align}
\end{cor}

\begin{exa} Set \( \Gamma = \Gamma(1/4) \). By \emph{Mathematica}, we have
\begin{align*}
\sum_{n=1}^{\infty} \frac{(-1)^{n}(2n-1)^{5}}{\sinh^{2}((2n-1)\pi/2) \cosh((2n-1)\pi/2)} &= \frac{15\Gamma^{10}}{256\sqrt{2}\pi^{17/2}} - \frac{3\Gamma^{12}}{256\pi^{9}}+\frac{7\Gamma^{14}}{2048\sqrt{2}\pi^{21/2}},\\
\sum_{n=1}^{\infty} \frac{(-1)^{n}(2n-1)^{9}}{\sinh^{2}((2n-1)\pi/2) \cosh((2n-1)\pi/2)}& = \frac{-3969\Gamma^{18}}{16384\sqrt{2}\pi^{29/2}} + \frac{189\Gamma^{20}}{4096\pi^{15}}-\frac{1809\Gamma^{22}}{131072\sqrt{2}\pi^{33/2}} ,\\
\sum_{n=1}^{\infty} \frac{(-1)^{n}(2n-1)^{13}}{\sinh^{2}((2n-1)\pi/2) \cosh((2n-1)\pi/2)} &= \frac{5756751\Gamma^{26}}{1048576\sqrt{2}\pi^{41/2}}-\frac{68607\Gamma^{28}}{65536\pi^{21}} +\frac{2630583\Gamma^{30}}{8388608\sqrt{2}\pi^{45/2}} .
\end{align*}
\end{exa}

\section{Evaluations of Berndt-type Integrals}

In this section, we provide the detailed proof of equation \eqref{mainformula} and present some illustrative examples.

\begin{thm}\label{thmEA}
 Let \( m \in \mathbb{N}\setminus \{1\} \) and \( \Gamma = \Gamma(1/4) \). Then, the following integral evaluates as
\begin{align}\label{6.1}
\int_{0}^{\infty} \frac{x^{4m-3} \mathrm{d}x}{[\cosh(2x) - \cos(2x)] [\cosh x - \cos x]} &= q_{1,m} \frac{\Gamma^{8m-8}}{\pi^{2m-2}} + {q_{2,m}} \frac{\Gamma^{8m-6}}{\sqrt{2}\pi^{2m-3/2}} \nonumber\\
&+ q_{3,m}\frac{\Gamma^{8m-4}}{\pi^{2m-1}} + {q_{4,m} \frac{\Gamma^{8m-2}}{\sqrt{2}\pi^{2m+1/2}}} + q_{5,m} \frac{\Gamma^{8m}}{\pi^{2m+2}},
\end{align}
where the constants \( q_{1,m}, q_{2,m}, q_{3,m}, q_{4,m}, q_{5,m} \in \mathbb{Q} \).
\end{thm}

\begin{proof}
Following the work of Rui-Xu-Zhao's paper \cite{RXZ2023}: for positive integers \( n \) and \( k \), define \( {P}_{n}^{(k)}:= {P}_{n}^{(k)}(1/2) \),\ \( {S}_{n}^{(k)}:= {S}_{n}^{(k)}(1/2) \),\ \(R_{n}^{(k)}=R_{n}^{(k)}(1 / 2)\) and \({A}_{n}^{(k)} = {A}_{n}^{(k)}(1/2) \).
From \cite[Eqs. (3.27)-(3.30)]{RXZ2023}, we know that \({R}_{4m-4}^{(0)}= {R}_{4m-2}^{(1)}= {R}_{4m-6}^{(1)}={R}_{4m-4}^{(2)}= 0 \). Then, for any integer \(m>1\) we have
\begin{equation} \label{BAREB}
\sum_{n=1}^{\infty} \frac{(-1)^{n} n^{4 m-3}}{\sinh ^{3}(n \pi)}=-\frac{(4 m-6) ! \Gamma^{8 m}}{2^{8 m+3} \pi^{6 m}}\left\{4\left(64(m-1)(4 m-3) \pi^{4} / \Gamma^{8}+m-3\right) R_{4 m-6}+R_{4 m-6}''\right\}.
\end{equation}
By substituting the equations \eqref{skyAE}, \eqref{skyDBA}, \eqref{skyDCC} and \eqref{BAREB} into Theorem \ref{thmCA}, we obtain
\begin{align}
& 2 \int_{0}^{\infty} \frac{x^{4m-3} \mathrm{d}x}{[\cosh(2x) - \cos(2x)] [\cosh x - \cos x]} \nonumber\\
&= (-1)^{m-1} \frac{1}{2^{6m}} (-8m^{2} + 14m - 6)(4m-6)!{R}_{4m-6} \frac{\Gamma^{8m-8}}{\pi^{2m-2}} \nonumber\\
&\quad + (-1)^{m} \frac{1}{2^{6m}} (4m-3)\left({A}_{4m-4} +{S}_{4m-4}\right)\frac{1}{\sqrt{2}}\frac{\Gamma^{8m-6}}{\pi^{2m-3/2}} \nonumber\\
&\quad + (-1)^{m-1} \frac{1}{2^{6m+2}} {P}_{4m-3}\frac{\Gamma^{8m-4}}{\pi^{2m-1}}\nonumber \\
&\quad + (-1)^{m} \frac{1}{2^{6m+3}} \left( -{S}_{4m-4}' + {S}_{4m-4} + {A}_{4m-4} +{A}_{4m-2}' \right) \frac{1}{\sqrt{2}} \frac{\Gamma^{8m-2}}{\pi^{2m+1/2}} \nonumber\\
&\quad + (-1)^{m-1} \frac{1}{2^{6m+7}} \left[ (4m-12){R}_{4m-6} + \text{R}_{4m-6}'' \right](4m-6)! \frac{\Gamma^{8m}}{\pi^{2m+2}}.
\end{align}
In particular, from this, the coefficients \( q_{i,m} \) can be explicitly identified as
\begin{align}
q_{1,m} &= (-1)^{m-1}\frac{1}{2^{6m}} (-8m^{2} + 14m - 6)(4m-6)!{R}_{4m-6},
\label{6.4}\\
q_{2,m}& =  (-1)^{m} \frac{1}{2^{6m}}(4m-3) \left( {A}_{4m-4}+{S}_{4m-4}\right),
\label{6.5}\\
q_{3,m} &= (-1)^{m-1} \frac{1}{2^{6m+2}} {P}_{4m-3},
\label{6.6}\\
q_{4,m} &= (-1)^{m} \frac{1}{2^{6m+3}} \left( -{S}_{4m-4}' + {S}_{4m-4} + {A}_{4m-4} +{A}_{4m-2}' \right) ,\label{6.7}\\
q_{5,m} &= (-1)^{m-1} \frac{1}{2^{6m+7}} \left[ (4m-12){R}_{4m-6} + {R}_{4m-6}'' \right](4m-6)! .
\label{6.8}
\end{align}
Thus, the proof of this theorem is complete.
\end{proof}

With the help of \emph{Mathematica}, we can perform the following computations by applying the formulas from Theorem \ref{thmEA}.
\begin{exa}
Let \( \Gamma = \Gamma(1/4) \). We have
\begin{flalign*}
 \int_{0}^{\infty} &\frac{x^{5} \mathrm{d}x}{[\cosh(2x) - \cos(2x)] [\cosh x - \cos x]} \\
&= \frac{5\Gamma^{8}}{1024\pi^{2}} - \frac{15\Gamma^{10}}{8192\sqrt{2}\pi^{5/2}} + \frac{3\Gamma^{12}}{16384\pi^{3}} - \frac{7\Gamma^{14}}{65536 \sqrt{2}\pi^{9/2}} + \frac{\Gamma^{16}}{65536\pi^{6}},&&
\end{flalign*}
\begin{flalign*}
 \int_{0}^{\infty} &\frac{x^{9} \mathrm{d}x}{[\cosh(2x) - \cos(2x)] [\cosh x - \cos x]} \\
&= \frac{81\Gamma^{16}}{16384\pi^{4}} - \frac{3969 \Gamma^{18}}{2097152\sqrt{2}\pi^{9/2}} + \frac{189\Gamma^{20}}{1048576\pi^{5}} - \frac{1809\Gamma^{22}}{16777216\sqrt{2}\pi^{13/2}} + \frac{17\Gamma^{24}}{1048576\pi^{8}},&&
\end{flalign*}
\begin{flalign*}
\int_{0}^{\infty}&  \frac{x^{13} \mathrm{d}x}{[\cosh(2x) - \cos(2x)] [\cosh x - \cos x]} \\
&= \frac{7371\Gamma^{24}}{262144\pi^{6}} - \frac{5756751\Gamma^{26}}{536870912\sqrt{2}\pi^{13/2}} + \frac{68607\Gamma^{28}}{67108864\pi^{7}} - \frac{2630583\Gamma^{30}}{4294967296\sqrt{2}\pi^{17/2}} + \frac{1539\Gamma^{32}}{16777216\pi^{10}}.&&
\end{flalign*}
\end{exa}

\section{Berndt-type Integrals via Barnes Multiple Zeta Functions}

In the present section, we assess the Barnes multiple zeta function by making use of outcomes from Berndt-type integrals.
Bradshaw and Vignat \cite{BV2024} offered integral representations for these functions, thereby retrieving a previously established result \cite[Eq. (3.2)]{R2000}.

\begin{pro}
Let \( \Re(s) > N \), \( \Re(\omega) > 0 \), and \( \Re(a_j) > 0 \) for \( j = 1, \ldots, N \). Then
\[
\zeta_N(s, \omega|a_1, \ldots, a_N) = \frac{1}{\Gamma(s)} \int_{0}^{\infty} u^{s-1} e^{-\omega u} \prod_{j=1}^{N} \frac{1}{(1 - e^{-a_j u})} \mathrm{d}u.
\]
The alternating form satisfies the following property:
\[
\tilde{\zeta}_N(s, \omega|a_1, \ldots, a_N) = \frac{1}{\Gamma(s)} \int_{0}^{\infty} u^{s-1} e^{-\omega u} \prod_{j=1}^{N} \frac{1}{(1 + e^{-a_j u})} \mathrm{d}u.
\]
\end{pro}
We can now employ the results on Berndt-type integrals established in the previous section to derive the following structural theorem for Barnes multiple zeta functions.
\begin{thm}\label{thmFB}
 For positive integer \( m>1 \), we have
\[
\Gamma(4m-2) \zeta_{4}(4m-2, 3|\mathbf{c}_4, \boldsymbol{\sigma}_4) \in \mathbb{Q} \frac{\Gamma^{8m-8}}{\pi^{2m-2}} + \frac{\mathbb{Q}}{\sqrt{2}} \frac{\Gamma^{8m-6}}{\pi^{2m-3/2}} + \mathbb{Q} \frac{\Gamma^{8m-4}}{\pi^{2m-1}} + \frac{\mathbb{Q}}{\sqrt{2}} \frac{\Gamma^{8m-2}}{\pi^{2m+1/2}} + \mathbb{Q} \frac{\Gamma^{8m}}{\pi^{2m+2}},
\]
where \( \mathbf{c}_4 = (2+2i, 2-2i, 1+i, 1-i) \) and \( \boldsymbol{\sigma}_4 = (\{1\}^4) \).
\end{thm}

\begin{proof}Following \cite[Prop. 2]{DCLMRT1992}, we obtain
\[
\begin{aligned}
& \int_{0}^{\infty} \frac{x^{s-1} e^{-\omega x}}{\prod_{i=1}^{M} \sinh(a_i x) \prod_{j=1}^{N} \cosh(b_j x)} \mathrm{d}x \\
&= \sum_{\substack{n_1, \ldots, n_M \geq 0 \\ k_1, \ldots, k_N \geq 0}} \frac{2^{M+N} \Gamma(s) (-1)^{k_1 + \cdots + k_N}}{\left( \omega + \sum_{i=1}^{M} a_i n_i + \sum_{j=1}^{N} b_j \left( 2k_j + \sum_{i=1}^{M} a_i n_i + \sum_{j=1}^{N} b_j k_j \right) \right)^s} \\
&= 2^{M+N} \Gamma(s) \zeta_{M+N}\left( s, \omega + \sum_{i=1}^{M} a_i + \sum_{j=1}^{N} b_j \middle| \mathbf{c}_{M+N}, \boldsymbol{\sigma}_{M+N} \right),
\end{aligned}
\]
where \( \mathbf{c}_{M+N} = (2a_1, \ldots, 2a_M, 2b_1, \ldots, 2b_N) \) and \( \boldsymbol{\sigma}_{M+N} = (\{1\}^M, \{-1\}^N) \).
Noting the fact that
\[
\begin{aligned}
& \int_{0}^{\infty} \frac{x^{4m-3} \mathrm{d}x}{[\cosh(2x) - \cos(2x)] [\cosh x - \cos x]}= 4 \Gamma(4m-2) \zeta_{4}(4m-2, 3|\mathbf{c}_4, \boldsymbol{\sigma}_4) \\
&= \frac{1}{4} \int_{0}^{\infty} \frac{x^{4m-3} \mathrm{d}x}{\sinh[(1+i)x] \sinh[(1-i)x] \sinh\left( \frac{1+i}{2}x \right) \sinh\left( \frac{1-i}{2}x \right)},
\end{aligned}
\]
where \( \mathbf{c}_4 = (2+2i, 2-2i, 1+i, 1-i) \) and \( \boldsymbol{\sigma}_4 = (\{1\}^4 )\). Lastly, the application of Theorem \ref{thmEA} gives the intended evaluation.
\end{proof}
By utilizing \emph{Mathematica's computational framework} together with Theorem \ref{thmFB}, we carry out explicit calculations for the integer values $m=2,3,4$.
\begin{exa}
Let \( \Gamma = \Gamma(1/4) \). We have
\[
\begin{aligned}
\zeta_{4}(6, 3|\mathbf{c}_4, \boldsymbol{\sigma}_4) &= \frac{\Gamma^{8}}{98304\pi^{2}} - \frac{\Gamma^{10}}{262144\sqrt{2}\pi^{5/2}} + \frac{\Gamma^{12}}{2621440\pi^{3}} - \frac{7\Gamma^{14}}{31457280\sqrt{2}\pi^{9/2}} + \frac{\Gamma^{16}}{31457280\pi^{6}}, \\
\zeta_{4}(10, 3|\mathbf{c}_4, \boldsymbol{\sigma}_4) &= \frac{\Gamma^{16}}{293601280\pi^{4}} - \frac{7\Gamma^{18}}{5368709120\sqrt{2}\pi^{9/2}} + \frac{\Gamma^{20}}{8053063680\pi^{5}} \\
&\quad - \frac{67\Gamma^{22}}{901943132160\sqrt{2}\pi^{13/2}}+ \frac{17\Gamma^{24}}{1522029235520\pi^{8}}, \\
\zeta_{4}(14, 3|\mathbf{c}_4, \boldsymbol{\sigma}_4) &= \frac{\Gamma^{24}}{885837004800\pi^{6}} - \frac{71\Gamma^{26}}{164926744166400\sqrt{2}\pi^{13/2}} + \frac{11\Gamma^{28}}{268005959270400\pi^{7}} \\
&\quad -\frac{97429\Gamma^{30}}{3962200101853593600\sqrt{2}\pi^{17/2}} + \frac{19\Gamma^{32}}{5159114715955200\pi^{10}},
\end{aligned}
\]
where \( \mathbf{c}_4 = (2+2i, 2-2i, 1+i, 1-i) \) and \( \boldsymbol{\sigma}_4 = (\{1\}^4) \).
\end{exa}

{\bf Conflict of Interest}
The authors declare no conflict of interest regarding the publication of this article.

{\bf Data Availability}
No new data were generated or analyzed in this study, and therefore data sharing is not applicable.

{\bf Use of AI Tools Declaration}
The authors confirm that no artificial intelligence (AI) tools were used in the creation of this work.

\end{document}